\documentclass[12pt]{article}
\input epsf.tex

\usepackage{graphicx}
\usepackage{amsmath,amsthm,amsfonts,amscd,amssymb,comment,eucal,latexsym,mathrsfs}
\usepackage{stmaryrd}
\usepackage[all]{xy}

\usepackage{epsfig}

\usepackage[all]{xy}
\xyoption{poly}
\usepackage{fancyhdr}
\usepackage{wrapfig}
\usepackage{epsfig}

\usepackage[pdftex]{hyperref}

\theoremstyle{plain}
\newtheorem{thm}{Theorem}[section]
\newtheorem{prop}[thm]{Proposition}
\newtheorem{lem}[thm]{Lemma}
\newtheorem{cor}[thm]{Corollary}

\theoremstyle{definition}
\newtheorem{defn}{Definition}
\theoremstyle{remark}
\newtheorem{remark}{Remark}
\newtheorem{example}{Example}

\newtheorem{problem}{Problem}

\newtheorem{notation}{Notation}

\advance \topmargin by -\headheight
\advance \topmargin by -\headsep
\evensidemargin \oddsidemargin
  \def\C{{\mathbb{C}}}           \def\N{{\mathbb{N}}}    \def\R{{\mathbb{R}}}        \def\Z{{\mathbb{Z}}}

 \def\cB{{\mathcal{B}}}                \def\cR{{\mathcal{R}}}        

      \def\hG{{\widehat{G}}}                   

     \def\tF{{\tilde{F}}}         \def\tO{{\tilde{O}}}           

     \def\tf{{\tilde{f}}}                    

\newcommand{\G}{\Gamma}
\newcommand{\Ga}{\Gamma}
\newcommand{\La}{\Lambda}
\newcommand{\Om}{\Omega}
\newcommand{\Si}{\Sigma}
\newcommand{\eps}{\epsilon}

\renewcommand\a{\alpha}
\renewcommand\b{\beta}
\renewcommand\d{\delta}
\newcommand\g{\gamma}

\newcommand\om{\omega}
\newcommand\s{\sigma}
\renewcommand\t{\theta}

\newcommand\Aut{\operatorname{Aut}}

\newcommand\dom{\operatorname{dom}}

\newcommand\Haar{\operatorname{Haar}}

\newcommand\injrad{\operatorname{injrad}}
\newcommand\Ind{\operatorname{Ind}}

\newcommand\Ker{\operatorname{Ker}}

\newcommand\rng{\operatorname{rng}}

\newcommand\sym{\operatorname{sym}}

\newcommand\vol{\operatorname{vol}}

\def\cc{{\curvearrowright}}

\begin{document}
\title{Locally compact sofic groups}
\author{Lewis Bowen\footnote{supported in part by NSF grant DMS-1900386}, Peter Burton\\ University of Texas at Austin}
\maketitle

\begin{abstract}
We introduce the notion of soficity for locally compact groups and list a number of open problems.
\end{abstract}

\noindent
{\bf Keywords}: sofic groups, metric approximation in group theory\\
{\bf MSC}:37A35\\

\noindent
\tableofcontents

\section{Introduction}

A countable discrete group $\G$ is {\bf sofic} if there exist maps $\s_i:\G \to \sym(V_i)$ (where $V_i$ is a finite set and $\sym(V_i)$ is its permutation group) satisfying 
\begin{eqnarray*}
1 &=&\lim_{i\to\infty} |V_i|^{-1} \{v \in V_i:~ \s_i(g)\s_i(h)v = \s_i(gh)v\} \quad \forall g,h \in \G\\
1 &=&\lim_{i\to\infty} |V_i|^{-1} \{v \in V_i:~ \s_i(g)v \ne v\} \quad \forall g \in \G \setminus \{1_\G\}.
\end{eqnarray*}
The first condition ensures that the maps behave asymptotically like homomorphisms and the second condition ensures that, asymptotically, every nontrivial element of $\G$ behaves like fixed point-free map on $V_i$. In this sense, we can think of $\s_i$ as providing a kind of approximation for the left-translation action of $\G$ on itself. 

Sofic groups were defined implicitly by M. Gromov in \cite{MR1694588} where he proved they satisfy Gottschalk's surjunctivity conjecture. Benjy Weiss made the subject more accessible by simplifying the proof of Gromov's result in \cite{weiss-2000}  and giving sofic groups their name, which is derived from the Hebrew word {\it sofi} meaning finite. For an introduction to sofic groups, see \cite{pestov-sofic-survey, pestov-kwiatkowska, capraro-lupini}. At the time of this writing, it is an open problem whether all discrete countable groups are sofic.


Locally compact sofic groups and their entropy theory were introduced by the first author and Sukhpreet Singh in Singh's unpublished 2016 thesis (available upon request to the first author). In this note, we give a new approach to locally compact sofic groups via partial actions and charts. Our main results are informally summarized as follows.
\begin{itemize}
\item Theorem \ref{T:unimodular}. Every sofic group is unimodular.
\item Theorem \ref{thm:amenable}. Every unimodular lcsc amenable group is sofic.
\item Theorem \ref{thm:injrad1}. A sequence of local $G$-spaces is a sofic approximation if and only if the essential injectivity radius of the sequence is infinite.
\item Theorem  \ref{T:discretegroup}. The new definition of sofic given in this paper generalizes the previous definitions for discrete countable groups.
\item Theorem \ref{T:soficlattice}. If $G$ admits a sofic lattice subgroup then $G$ is sofic. 
\item Corollary \ref{C:lattices}. The following groups are sofic: semi-simple Lie groups (e.g. $\rm{SL}(n,\R), \rm{SO}(n,1)$ etc), the automorphism group of a regular tree. 
\item Proposition \ref{P:open}. If $G$ is sofic and $H\le G$ is an open subgroup then $H$ is sofic.
\end{itemize}

This paper is organized as follows. Fix a  locally compact second countable group $G$. In \S 2 we introduce local $G$-spaces as  topological spaces with a partial homogeneous action of $G$. We derive metric and measure-theoretic properties of these spaces. In \S 3 we define sofic approximations to $G$ as sequences of local $G$-spaces which, in a sense, approximate the action of $G$ on itself by right-translations. We also prove the main results. The last section gives a series of open problems. 


\section{Local $G$-spaces and partial actions}

\subsection{Local $G$-spaces} \label{subsec.localgspace}

We use the abbreviation lcsc to mean locally compact second countable. Let $G$ be an lcsc group.

\begin{defn}\label{D:partial}
A {\bf partial right-action} of $G$ on a Hausdorff space $M$ is a continuous map $\a:\dom(\a) \to M$ where $\dom(\a) \subset M \times G$ is open.  We require the following axioms hold for all $p\in M$.
\begin{enumerate}
\item[Axiom 1.] $(p,1_G) \in \dom(\a)$ and $\a(p,1_G) = p$. \label{D:partial-identity}
\item[Axiom 2.] If $(p,g)\in \dom(\a)$ then $(\a(p,g),g^{-1}) \in \dom(\a)$ and $\a( \a(p,g), g^{-1}) = p$.  \label{D:partial-inverse}
\item[Axiom 3.] If $(p,g), (\a(p,g),h), (p,gh) \in \dom(\a)$ then $\a(p,gh)=\a(\a(p,g),h)$. \label{D:partial-mult}
\end{enumerate}
A partial action $\a$ is {\bf homogeneous} if in addition it satisfies the following.
\begin{enumerate}
\item[Axiom 4.] \label{D:partial-homeo} For every $p\in M$ there is an open neighborhood $O_p$ of  $1_G$ in $G$ such that $\{p\} \times O_p \subset \dom(\a)$ and the  
restriction of $\a(p,\cdot)$ to $\{p\} \times O_p$ is a homeomorphism onto an open neighborhood of $p$ in $M$. 
\end{enumerate}

\end{defn}

\begin{defn}
A {\bf local $G$-space} is a pair $(M,\a)$ where $M$ is an lcsc space and $\a$ is a partial homogeneous right-action. 
\end{defn}

\begin{notation}
We will usually denote a local $G$-space by $M$ (or $V$), leaving the action $\a$ implicit. To simplify notation, we write $p.g= \a(p,g)$. If $K \subset M$, we will also write $K.g = \{ \a(k,g):~k \in K\}$. In particular, $K.g$ is well-defined if and only if $K\times \{g\}$ is in the domain of the action $\a$. Similarly, we write $p.O = \{\a(p,g):~g\in O\}$ if $\{p\}\times O \subset \dom(\a)$. 
\end{notation}

\begin{remark}
By Axiom 3, $p.g_1.g_2=p.g_1g_2$ when both sides are defined. However this does not imply that $p.g_1.g_2.g_3 = p.g_1g_2g_3$ even when both sides are defined. See example \ref{E:branched}.
\end{remark}

\begin{lem}\label{L:injective}
Let $M$ be a local $G$-space. Let $g \in G$. Then $\a(\cdot, g)$ is injective (where it is defined).
\end{lem}

\begin{proof}
Suppose $\a(p,g)=\a(q,g)$ for some $p,q \in M$. In other words, $p.g=q.g$. By Axiom 2 of Definition \ref{D:partial}, $p.g.g^{-1}=p$ and $q.g.g^{-1}=q$ and both are well-defined. Therefore, $p=q$ as required.

\end{proof}

\begin{cor}
Let $M$ be a local $G$-space. Let $K \subset M$ be compact, $g \in G$ and suppose $K.g$ is well-defined. Then the map $\a(\cdot, g)$ restricted to $K$ is a homeomorphism onto its image.
\end{cor}

\begin{proof}
Because $K$ is compact, it suffices to prove the map is continuous and injective. Continuity follows from joint continuity of $\a$ and injectivity follows from the previous lemma.
\end{proof}

\subsection{Charts}

\begin{defn}
Let $(M,\a)$ be a local $G$-space and $p \in M$. A {\bf chart centered at $p$} is a homeomorphism $f_p:\dom(f_p) \to \rng(f_p)$ where $\dom(f_p) \subset M$ is an open neighborhood of $p$, $\rng(f_p)$ is an open neighborhood of the identity in $G$ and $g=f_p(p.g)$ for all $g \in \rng(f_p)$. By Axiom 4 of Definition \ref{D:partial}, for every $p \in M$ there exists a chart centered at $p$.  
\end{defn}

We will show that the transition functions between two charts are locally given by left-translation in $G$. 

\begin{defn}
Let $A, B \subset G$ be Borel sets. A map $\phi:A \to B$ is {\bf locally left-translation} if there exists a decomposition $A = \sqcup_{i \in I} A_i$ into relatively open sets and $\{g_i\}_{i\in I} \subset G$ (for some countable index set $I$) such that $\phi(a)=g_ia$ for all $a\in A_i$. By relatively open we mean $A_i$ is open in $A$.  
\end{defn}

\begin{prop}\label{P:chart}
Let $(M,\a)$ be a local $G$-space. Let $p,q \in M$ and let $f_p,f_q$ be charts centered at $p,q$ respectively. Let 
$$A=\{g \in \rng(f_p):~ p.g \in  \dom(f_p) \cap \dom(f_q)\}, \quad B = \{g \in \rng(f_q):~ q.g \in  \dom(f_p) \cap \dom(f_q)\}.$$
Then there is a map $\tau:A \to B$ which is locally left-translation such that $f_q(r) = \tau( f_p(r))$ for all $r \in \dom(f_p) \cap \dom(f_q)$. Moreover, $\tau$ is bijective and left-Haar-measure-preserving. 
\end{prop}

\begin{proof}

Let $r \in  \dom(f_p) \cap \dom(f_q)$. Let $g \in \rng(f_p)$ be such that $p.g=r$. Let $h \in \rng(f_q)$ be such that $q.h=r$. Choose a chart $f_r$ centered at $r$. After choosing $\rng(f_r)$ smaller if necessary we may assume $g\rng(f_r) \subset \rng(f_p)$ and $h \rng(f_r) \subset \rng(f_q)$. This implies $r.k = p.g.k = p.gk$ for all $k \in \rng(f_r)$. Similarly, $r.k = q.hk$. 

Let $C_r =  hg^{-1}$. We claim that $f_q(r.k)=C_r f_p(r.k)$ for all $k \in \rng(f_r)$. This follows from
$$f_q(r.k) = f_q(p.hk) = hk =  C_r g k = C_r f_p(p.gk) = C_r f_p(r.k).$$
Since $r.\rng(f_r) =\dom(f_r)$, this implies $f_q(s)=C_r f_p(s)$ for all $s\in \dom(f_r)$. 

Because $G$ is lcsc, it follows that there are a countable index set $I$, $\{r_i\}_{i\in I} \subset  \dom(f_p) \cap \dom(f_q)$ such that 
$$ \dom(f_p) \cap \dom(f_q) \subset \cup_{i\in I} \dom(f_{r_i})$$
where $f_{r_i}$ are as above.

Let $A, B \subset G$ be as in the statement. Define  $\tau:A \to B$ by $\tau(g) = C_{r_i} g$ if $p.g \in \dom(f_{r_i})$. By the previous paragraph $\tau(f_p(r))=f_q(r)$ for all $r\in  \dom(f_p) \cap \dom(f_q)$. In particular since $f_p$ and $f_q$ are homeomorphisms, $\tau$ is well-defined, bijective and continuous. Since $g \mapsto \tau(g)g^{-1}$ is locally constant, $\tau$ is locally left-translation and therefore it is left-Haar-measure-preserving.
\end{proof}

\subsubsection{Measures}

Here we show that a local $G$-space admits a canonical measure.

\begin{prop}[The canonical measure]
Let $(M,\a)$ be a local $G$-space. Fix a left-Haar measure $\Haar_G$ on $G$. Then there exists a unique Radon measure $\vol_M$ on $M$ satisfying the following. If $p \in M$, $f_p$ is a chart centered at $p$ and $K \subset \dom(f_p)$ is Borel then 
\begin{eqnarray}\label{E:vol}
 \vol_M(K) = \Haar_G(\{g \in \rng(f_p):~ p.g  \in K\}) = \Haar_G(f_p(K)).
 \end{eqnarray}
\end{prop}

\begin{proof}
Let $K \subset M$ be Borel. Because $M$ is lcsc there exist a countable index set $I$, charts $\{f_i\}_{i\in I}$ with $f_i$ centered at $p_i \in M$ such that $K \subset \cup_{i\in I} \dom(f_i)$. Therefore, there is a Borel decomposition $K = \sqcup_{i\in I} K_i$ with $K_i \subset \dom(f_i)$ for all $i$. We define
$$\vol_M(K) = \sum_{i\in I} \Haar_G(f_i(K_i)).$$
In order to show this is well-defined, suppose that $J$ is a countable index set, $\{g_j\}_{j\in J}$ are charts, $K = \sqcup_{j \in J} L_j$ is a Borel decomposition and $L_j \subset \dom(g_j)$. We must show
$$\sum_{i\in I} \Haar_G(f_i(K_i)) = \sum_{j\in J} \Haar_G(g_j(L_j)).$$
By countable additivity, it suffices to show that 
$$\Haar_G(f_i(K_i \cap L_j)) =  \Haar_G(g_j(K_i \cap L_j)).$$
for all $i \in I, j \in J$. This follows from Proposition \ref{P:chart}.

This shows that $\vol_M$ is well-defined and satisfies $\vol_M(K) =  \Haar_G(f_p(K))$ whenever $f$ is a chart with $K \subset \dom(f)$. Because $f_p$ is a measure-preserving homeomorphism and $G$ is locally compact, it follows that $\vol_M$ is a Radon measure. 

\end{proof}

\begin{defn}
Let $\d:G \to \R_{>0}$ be the {\bf modular function}. This means that if $S \subset G$ has finite Haar measure then $\Haar_G(Sg)=\d(g)\Haar_G(S)$ where $\Haar_G$ is a left-Haar measure on $G$. The modular function is a homomorphism. $G$ is {\bf unimodular} if $ \d(g)=1$ for all $g\in G$. This means that Haar measure on $G$ is both left and right $G$-invariant.
\end{defn}

\begin{lem}[Locally measure-preserving]\label{L:mp}
Let $(M,\a)$ be a local $G$-space and suppose $K \times \{g\} \subset \dom(\a)$ for some measurable $K \subset M$ and $g \in G$. Then 
$$\vol_M(K.g) = \d(g)\vol_M(K).$$
\end{lem}

\begin{proof}
Let $p\in K$. Since $(p,g)\in \dom(\a)$ and $\dom(\a)$ is open in $M\times G$, there are open neighborhoods $U,V$ of $1_G$ in $G$ such that $\{p\} \times U \subset \dom(\a)$, $\{p\}\times Vg \subset \dom(\a)$. By Axiom 4 of Definition \ref{D:partial}), there is an open neighborhood $O_p$ of the identity in $G$ such that $\a(p,\cdot)$ restricts to a homeomorphism from $O_p$ to an open neighborhood $p.O_p$ of $p$ in $M$. After intersecting $O_p$ with $U$ and $V$ if necessary, we may assume $O_p \subset U \cap V$. 

For every $h \in O_p$, if $p.h \in K$ then $p.h$, $p.h.g$, $p.hg$ are all well-defined (and therefore $p.h.g=p.hg$ by Axiom 3 of Definition \ref{D:partial}).

Since $M$ is lscs, there exists a countable subset $\{p_i\}_{i \in I} \subset K$ and open neighborhood $O_i=O_{p_i}\subset G$ of the identity as above such that 
$$\bigcup_{i\in I} p_i.O_{i}  \supset K.$$
So there is a  measurable partition $K = \sqcup_{i=1}^\infty K_i$ such that for each $i$, $K_i \subset p_i.O_i$. Because the map $\a(\cdot, g):K \to M$ is injective (by Lemma \ref{L:injective}) and
$$\vol_M(K.g) = \sum_{i=1}^\infty \vol_M(K_i.g)$$
we may assume without loss of generality $K \subset p.O_p$ for some $p \in K$. 

We claim that $f_p(K.g)=f_p(K)g$ (and both sides are well-defined). To see this, let $k \in K \subset p.O_p$. Then there is $h \in O_p$ such that $k=p.h$. By choice of $O_p$,  $p.h.g=p.hg$ and both are well-defined. Thus 
$$f_p(k.g)=f_p(p.hg)=hg = f_p(p.h)g = f_p(k)g.$$
Since $k \in K$ is arbitrary, this proves the claim. So
$$\Haar_M(K.g) = \Haar_G( f_p(K.g) ) = \Haar_G( f_p(K)g) = \d(g)\Haar_G( f_p(K))  = \d(g)\vol_M(K).$$
The first and last equalities hold by definition of $\vol_M(\cdot)$. 
\end{proof}

\subsection{Metrics}

In this section, we show that, given a left-invariant proper metric $d_G$ on $G$, there is a canonical induced metric $d_M$ on any local $G$-space $M$. 

\begin{defn}
A local $G$-space $M$ is {\bf transitive} if for every $p,q \in M$ there exist elements $g_1,\ldots, g_n \in G$ such that $q= p.g_1.\cdots. g_n$. Equivalently, this means there exist charts $f_1,\ldots, f_n$ such that $p \in \dom(f_1), q \in \dom(f_n)$ and $\dom(f_i)\cap \dom(f_{i+1}) \ne \emptyset$ for all $i$. 
\end{defn}

\begin{defn}\label{D:injrad1}
Let $G$ be an lcsc group with a left-invariant proper metric $d_G$. Let $M$ be a local $G$-space. Let $B(\rho)\subset G$ be the open ball of radius $\rho$ centered at $1_G$. For $p\in M$, let $\injrad(M,p)$ be the supremum over all $\rho>0$ such that
\begin{enumerate}
\item for any $g,h \in G$ with $g, h, gh \in B(\rho)$, $p.g.h = p.gh$ (in particular, both sides are well-defined);
\item the restriction of $\a(p,\cdot)$ to $B(\rho)$ is a homeomorphism onto its image.
\end{enumerate}
This is the {\bf injectivity radius} at $p$.
\end{defn}

\begin{thm}\label{T:metric}
Let $d_G$ be a left-invariant proper metric on $G$. Let $M$ be a transitive left-$G$-space. Then there is a metric $d_M$ on $M$ satisfying the following local condition. For all $p \in M$, if $\rho=\injrad(M,p)$ and $g \in B(\rho)$ then $d_M(p,p.g)=d_G(1_G,g)$. 

\end{thm}

\begin{proof}
Let $p,q \in M$. Define 
$$d_M(p,q) = \inf \sum_{i=1}^{n-1} d_G(1_G, g_i)$$ 
where the infimum is over all sequences $g_1,\ldots, g_n \in G$ such that $p.g_1.\cdots.g_n = q$.


Let $p\in M$. Let $\rho=\injrad(M,p)$. To finish the proof it suffices to show that if $g \in B(\rho)$ then $d_M(p, p.g) = d_G(1_G,g)$. 

So suppose $h_1,\ldots, h_n \in G$ and $p.g = p.h_1.\cdots.h_n$. We must show 
\begin{align}\label{E:thing10}
d_G(1_G,g) \le  \sum_{i=1}^{n} d_G(1_G, h_i).
\end{align}
If for every $i$, $h_1\cdots h_i \in B(\rho)$ then (\ref{E:thing10}) is immediate by the triangle inequality. So we assume there is some index $m$ such that $h_1\cdots h_{m+1} \notin B(\rho)$. We may assume $m$ is the smallest index for which this holds. Thus $h_1\cdots h_i \in B(\rho)$ for all $i \le m$ and $h_1\cdots h_{m+1} \notin B(\rho)$. 

Because $d_G$ is left-$G$-invariant,
$$d_G(1_G,h_1) + d_G(1_G,h_2) = d_G(1_G,h_1) + d_G(h_1,h_1h_2) \ge d_G(1_G, h_1h_2).$$
By an inductive argument we obtain 
$$ \sum_{i=1}^{m+1} d_G(1_G, h_i) \ge d_G(1_G, h_1\cdots h_{m+1}) \ge \rho > d_G(1_G,g).$$
This proves (\ref{E:thing10}). 
 

 \end{proof}

\subsection{Examples}

\begin{example}\label{ex:homog}
If $\G < G$ is discrete then $\G \backslash G$ admits a local $G$-space structure as follows. Define $\a:\G \backslash G \times G \to \G \backslash G$ by $\a(\G g,h) = \G gh$. 

Fix $g \in G$ and let $\pi:G \to \G \backslash G$ be the map $\pi(h) = \G gh$. Since $\G$ is discrete, $\pi$ is a covering space map. In particular, it is a local homeomorphism. This implies Axiom 4 of Definition \ref{D:partial}. The other Axioms are immediate. 
\end{example}

\begin{example}\label{E:subset}
Let $M \subset G$ be an open subset. Let $\dom(\a)=\{(p,g)\in M \times G:~pg \in M\}$ and define $\a:\dom(\a) \to M$ by $\a(p,g)=pg$. It is immediate that $(M,\a)$ is a local $G$-space. Moreover, $\vol_M$ is the restriction of $\Haar_G$ to $M$. 
\end{example}

\begin{example}\label{E:branched}
In this example, we show that it is possible for $p.g.h.k \ne p.ghk$ even when both sides are well-defined. Let $G=\C$ be the complex plane, as an additive group. Let $M$ be a double cover of $\C\setminus \{0\}$. To be precise, let $M=(0,\infty) \times \R/4\pi \Z$. For $a, b \in \R/4\pi\Z$, define
$$|a - b| = \inf_{n \in \Z} |a' - b' - 4\pi n|$$
where $a', b' \in \R$ satisfy $a' = a \mod 4\pi \Z$ and $b' = b \mod 4\pi \Z$.

Let 
$$\dom(\a)=\{ ( (r,\t), se^{i\phi} - re^{i\t} ) \in M \times \C:~ s> 0,~ |\phi -\t| < 2\pi/3\}.$$
Define an action $\a:\dom(\a) \to M$ by
$$\a( (r,\t), z) = (t, \phi)$$
where $te^{i\phi} = re^{i\t} + z$ and $\phi \in \R/4\pi \Z$ is chosen to minimize $|\t - \phi|$. Note there are only two different elements $\phi_1, \phi_2 \in \R/4\pi\Z$ that satisfy $te^{i\phi_j} = re^{i\t} + z$ (for $j=1,2$) and $\phi_1-\phi_2 = 2\pi \mod 4\pi\Z$. Because these elements are $2\pi$ apart, at most one of them can be within $2\pi/3$ of $\t$. Moreover, the definition of $\dom(\a)$ shows that exactly one of these elements is within $2\pi/3$ of $\t$. So $\a$ is well-defined. 

Next we check that $(M,\a)$ is a local $G$-space. Define 
$$\tf: M \to \C, \quad \tf(r,\t) = re^{i\t}.$$
Then $\tf$ is a 2-1 covering map of $\C \setminus \{0\}$. Also $\tf(\a(p, z)) = \tf(p) + z$. Because $\tf$ is a local homeomorphism, $\a$ satisfies Axiom 4 of Definition \ref{D:partial}. 

In order to check Axiom 3, suppose $(r,\t).z_1, (r,\t).z_1.z_2$ and $(r,\t).(z_1+z_2)$ are all well-defined. We must show  $(r,\t).z_1.z_2=(r,\t).(z_1+z_2)$.  Let $(r_1,\t_1) = (r,\t).z_1$, $(r_2,\t_2) = (r,\t).z_1.z_2$ and $(r_3,\t_3) = (r,\t).(z_1+z_2)$. 

The assumption that these are all well-defined implies $|\t - \t_1|<2\pi/3$, $|\t_1-\t_2|< 2\pi/3$ and $|\t-\t_3|<2\pi/3$. The triangle inequality implies $|\t_2 - \t_3| < 2\pi$. On the other hand, the definition of $\a$ implies that either $\t_2 = \t_3$ or $\t_2 = \t_3 + 2\pi$ (mod $4\pi$). So we must have $\t_2 = \t_3 \mod 4\pi$. Therefore Axiom 3 holds. The other Axioms are immediate. 

Next we show the existence of $p \in M$ and $g,h,k \in G$ such that $p.g.h.k \ne p.ghk$ even though both sides are well-defined. Let $p=(1,0) \in M$. Let $g=i -1$, $h=-i-1$, $k = 1-i$. Then $\tf(p)+g = 1+(i-1)=i$. So $p.g = (1,\pi/2)$. Also $\tf(p.g) + h = i + h = -1$. So $p.g.h = (1,\pi)$. Finally, $\tf(p.g.h)+ k = -1 + k = -i$. So $p.g.h.k = (1,3\pi/2)$. On the other hand, $g+h+k = -i-1$. So $p.ghk = (1,-\pi/2)$. Because $-\pi/2 \ne 3\pi/2 \mod 4\pi$, we have $p.g.h.k \ne p.ghk$. 

\end{example}

\section{Sofic groups}

\subsection{Definitions}

\begin{defn}\label{D:sofic group}
Let $M=(M,\a)$ be a local $G$-space and let $U \subset G$ be open and pre-compact and let $\eps>0$.  Let $M[U]=M[\a,U]$ be the set of all $p \in M$ such that if $g,h \in G$ are such that $g, h, gh \in U$ then $p.g.h=p.gh$ (in particular, both sides are well-defined). Moreover, we require that the map $g \mapsto \a(p,g)$ is a homeomorphism from $U$ to an open neighborhood of $p$. We say $M$ is a {\bf $(U,\eps)$-sofic approximation to $G$} if $\vol_M(M)<\infty$ and
$$\vol_M(M[U]) \ge (1-\eps) \vol_M(M).$$
\end{defn}


\begin{defn} A \textbf{sofic approximation} to $G$ is a sequence $\Sigma = (M_i)_{i=1}^\infty$ where $M_i$ is a $(U_i,\epsilon_i)$-sofic approximation such that the $U_i$ are pre-compact open sets increasing to $G$ and the sequence $\epsilon_i$ decreases to $0$. We say $G$ is {\bf sofic} if it admits a sofic approximation. 
\end{defn}

The following lemma will be generally helpful.

\begin{lem}\label{L:restriction100}
If $U\subset G$ is a pre-compact open neighborhood of the identity and $g \in U$ then $M[U].g \subset M[U \cap g^{-1}U]$.
\end{lem}

\begin{proof}
Let $p\in M[U]$, $g \in U$ and $h, k \in G$ be such that $h, k, hk \in U \cap g^{-1}U$. To show $p.g \in M[U \cap g^{-1}U]$, we first show $p.g.h.k = p.g.hk$. 

Because $p \in M[U]$, and $g,h,gh \in U$ we have $p.g.h=p.gh$. Because $gh, k, ghk \in U$ we have $p.gh.k = p.ghk$. Combine these equalities to obtain $p.g.h.k = p.g.hk$.

Next define $\b:U \cap g^{-1}U \to M$ by $\b(k) = p.g.k$. We must show $\b$ is a homeomorphism onto an open neighborhood of $p.g$. Let $\g:U \to M$ be the map $\g(k) = p.k$. This map is a homeomorphism onto an open neighborhood of $p$ since $p \in M[U]$. Moreover, $\b(k) = p.g.k = p.gk = \g(gk)$. So $\b$ is the composition of left multiplication by $g$ with $\g$. This implies it has the claimed properties.


\end{proof}

\subsection{Unimodularity}


The goal of this section is to prove:
\begin{thm}\label{T:unimodular}
If $G$ is sofic then $G$ is unimodular.
\end{thm}

\begin{proof}
Fix a left-Haar measure $\Haar_G$ on $G$. Let $\eps>0$, $g \in G$, and $U \subset G$ be a pre-compact open set containing $\{1_G,g, g^2\}$. Suppose $M$ is a $(U \cup g^{-1}U,\eps)$-sofic approximation.  


By Lemma \ref{L:restriction100}, $M[g^{-1}U].g^{-1} \subset M[U \cap g^{-1}U]$. Multiply by $g$ on the right to obtain $M[U\cap g^{-1}U].g \supset M[g^{-1}U]$. Thus 
$$\vol_M(M[U\cap g^{-1}U].g) \ge \vol_M(M[g^{-1}U]) \ge \vol_M(M[U\cup g^{-1}U]) \ge (1-\eps)\vol_M(M).$$

Let $\d:G \to \R$ be the modular function. Lemma \ref{L:mp} implies 
$$\vol_M(M[U\cap g^{-1}U].g)  = \d(g) \vol_M(M[U\cap g^{-1}U]).$$
Combining this with the previous inequality, we obtain
$$\vol_M(M) \ge \vol_M(M[U\cap g^{-1}U]) \ge \d(g)^{-1}(1-\eps)\vol_M(M).$$
Therefore, $\d(g) \ge 1-\eps$. Since this is true for every $\eps>0$, $\d(g)\ge 1$. However $\d:G \to \R_{>0}$ is a homomorphism. So $\d(g^{-1}) = \d(g)^{-1}$. Since also $\d(g^{-1})\ge 1$, we obtain $\d(g)=1$. Because $g\in G$ is arbitrary, $G$ must be unimodular. 

\end{proof}

\subsection{Amenable groups}

\begin{thm}\label{thm:amenable}
Every unimodular amenable lcsc group is sofic.
\end{thm}

\begin{defn}
A locally compact group $G$ is {\bf amenable} if for every left-Haar measure $\Haar_G$, there exists a sequence $\{F_i\}_{i=1}^\infty$ of measurable sets with finite positive measure  such that for every compact $K \subset G$,
\begin{align}\label{Amenable0}
\lim_{i\to \infty} \frac{\Haar_G(KF_i)}{\Haar_G(F_i)}=1.
\end{align}
Such a sequence is called a {\bf left-F\o lner sequence}. 
\end{defn}

\begin{lem}\label{L:Folner}
If $G$ is amenable and unimodular then there exists a sequence $\{\tF_i\}_{i=1}^\infty$ of finite measure subsets such that each $\tF_i$ is pre-compact, open and for every compact $K \subset G$,
\begin{align}\label{Amenable}
\lim_{i\to \infty} \frac{\Haar_G(\{p \in \tF_i:~ pK \subset \tF_i\})}{\Haar_G(\tF_i)}=1.
\end{align}
\end{lem}

\begin{proof}

Let $\{F_i\}_{i=1}^\infty$ be a left-F\o lner sequence.  After perturbing slightly, we may assume each $F_i$ is pre-compact. Indeed, since $F_i$ has finite measure, there is a subset $L_i \subset F_i$ which is pre-compact such that $\lim_{i\to\infty}  \frac{\Haar_G(F_i\setminus L_i)}{\Haar_G(F_i)}=0$. Then $\{L_i\}_{i=1}^\infty$ is left-F\o lner. So we can replace $F_i$ with $L_i$. 

Let $\{K_i\}_{i=1}^\infty$ be an increasing sequence of pre-compact open subsets with $G = \cup_i K_i$. We choose $K_i$ to grow so slowly that
\begin{align}\label{Amenable1}
\lim_{i\to \infty} \frac{\Haar_G(K_iF_i)}{\Haar_G(F_i)}=1.
\end{align}
Let $F'_i= K_iF_i$. Then $F'_i$ is pre-compact and open. Because $\lim_{i\to \infty}\frac{\Haar_G(F'_i \vartriangle F_i)}{\Haar_G(F_i)}=0$,  $\{F'_i\}_{i=1}^\infty$ is left-F\o lner. 

For any compact $K \subset G$, there exists $i$ with $K \subset K_i$. Therefore,
\begin{eqnarray}
\liminf_{i\to \infty} \frac{\Haar_G(\{p \in F'_i:~  Kp \subset F'_i\})}{\Haar_G(F'_i)} &\ge& \liminf_{i\to \infty} \frac{\Haar_G(\{p \in F'_i:~ K_ip \subset F'_i\})}{\Haar_G(F'_i)}\nonumber \\ 
&\ge& \liminf_{i\to \infty} \frac{\Haar_G(F_i)}{\Haar_G(F'_i)} =1. \label{E:it}
\end{eqnarray}

Now let $\tF_i = (F'_i)^{-1}$. Note $\tF_i$ is pre-compact and open since $F'_i$ is. Let $K \subset G$ be compact. Suppose $p \in \tF_i$ satisfies $pK \subset \tF_i$. Then $K^{-1}p^{-1} \subset F'_i$. The converse is also true. Thus
$$\{p \in \tF_i:~  pK \subset \tF_i\}^{-1} = \{p \in F'_i:~  K^{-1}p \subset F'_i\}.$$
Because $G$ is unimodular, $\Haar_G(E)=\Haar_G(E^{-1})$ for any measurable $E \subset G$. Thus
\begin{eqnarray*}
\liminf_{i\to \infty} \frac{\Haar_G(\{p \in \tF_i:~  pK \subset \tF_i\})}{\Haar_G(\tF_i)} &=& \liminf_{i\to \infty} \frac{\Haar_G(\{p \in F'_i:~  K^{-1}p \subset F'_i\})}{\Haar_G(F'_i)} = 1
\end{eqnarray*}
by (\ref{E:it}).

\end{proof}

\begin{proof}[Proof of Theorem \ref{thm:amenable}]
Let $G$ be a unimodular amenable lcsc group. Let $\{F_i\}_{i=1}^{\infty}\subset G$ be a sequence as in Lemma \ref{L:Folner}. As in Example \ref{E:subset} we may regard $F_i$ as a local $G$-space with $\vol_{F_i}$ equal to the restriction of $\Haar_G$ to $F_i$ where (as always in this paper) $\Haar_G$ is a left-Haar measure.  Let  $U \subset G$ be pre-compact and $\eps>0$. Then $F_i[U]$ consists of all $p\in F_i$ with $pU \subset F_i$. By Lemma \ref{L:Folner}, 
$$\vol_{F_i}(F_i[U]) = \Haar_G(\{p \in F_i:~pU \subset F_i\}) \ge (1-\eps) \Haar_G(F_i)=(1-\eps) \vol_{F_i}(F_i)$$
for all sufficiently large $i$. Thus $\{F_i\}_{i=1}^{\infty}$ is a sofic approximation and $G$ is sofic.
\end{proof}

\subsection{The metric approach to soficity}

For this section, fix an lcsc group $G$ with a left-invariant proper metric $d_G$.
\begin{thm}\label{thm:injrad1}
Let $(M_i)_{i=1}^\infty$ be a sequence of local $G$-spaces.  Then the following are equivalent.
\begin{enumerate}
\item $\{M_i\}_{i=1}^\infty$ is a sofic approximation to $G$.
\item For every $\rho>0$
$$\lim_{i\to\infty} \frac{ \vol_{M_i}( \{p \in M_i:~ \injrad(M_i,p)>\rho \}) }{\vol_{M_i}(M_i)} = 1$$
\end{enumerate}
where injectivity radius is defined in Definition \ref{D:injrad1}.
\end{thm}

\begin{proof}
Let $U=B(\rho)$, which is the open ball of radius $\rho$ centered at the identity in $G$. Then $\injrad(M_i,p)\ge \rho$ if and only if $p\in M_i[U]$. So the theorem follows immediately from the definition of sofic approximation.
\end{proof}

\begin{remark}
It is possible to define Benjamini-Schramm (BS) convergence for measured metric spaces via the pointed Gromov-Hausdorff-Prokhorov topology \cite{MR3314481}. Theorem \ref{thm:injrad1} implies that if a sequence of local $G$-spaces equipped with metrics given by Theorem \ref{T:metric} is a sofic approximation to $G$ then it BS-converges to $G$. The converse is not true. For example, there are non-isomorphic finitely generated groups $G_1,G_2$ with Cayley graphs which are isomorphic as unlabeled graphs.  Precisely, there are generating sets $S_i$ for $G_i$ ($i=1,2$) so that with the corresponding word-metrics there is an isometry $\phi:G_1 \to G_2$ that maps the identity to the identity and $S_1$ bijectively onto $S_2$.

Suppose $\{M_i\}$ is a sofic approximation to $G_1$. Then $\{M_i\}$ BS-converges to $G_1$. However, we can think of $\{M_i\}$ as a local $G_2$-space in the following way. For $p \in M_i$ and $g \in S_2 \cup \{1_{G_2}\}$, let $p.g$ equal $p.\phi^{-1}(g)$ (if and only if the latter is defined). Because $G_2$ is finitely generated, this makes each $M_i$ into a local $G_2$-space. With this structure, $\{M_i\}$ is not a sofic approximation to $G_2$ (since $G_1$ and $G_2$ are non-isomorphic). But it does BS-converge to $G_2$ (since $G_1$ and $G_2$ are isometric).



\end{remark}


\subsubsection{Discrete sofic groups}

In this section we show that our new definition of sofic agrees with the standard definition if $G$ is a countable discrete group.

\begin{defn}\label{D:discretesofic}
Let $G$ be a countable group.  Let $\s:G \to \sym(V)$ (where $V$ is a finite set) be a set map. For $U \subset G$ let $V[\s,U] \subset V$ be the set of all $p$ such that 
\begin{enumerate}
\item $\s(g)\s(h)p = \s(gh)p$ for all $g,h \in U$;
\item $\s(g)p \ne \s(h)p$ for all $g,h \in U$ with $g \ne h$. 
\end{enumerate}
Then $\s$ is a  {\bf discrete $(U,\eps)$-sofic approximation} if
$$\#V[\s,U] \ge (1-\eps)\#V.$$
 $G$ is {\bf sofic as a discrete group} if for every finite $U \subset G$ and $\eps>0$ there exists a $(U,\eps)$-sofic approximation to $G$. 
\end{defn}

We first show that if $G$ is sofic as a discrete group then it admits a sofic approximation that is exact with respect to inverses and the identity.
\begin{lem}\label{L:trivial}
If $G$ is sofic as a discrete group then for every finite $U \subset G$ and $\eps>0$ there exists a discrete $(U,\eps)$-sofic approximation $\s:G \to \sym(V)$ such that $\s(1_G)$ is the identity and $\s(g^{-1})=\s(g)^{-1}$ for all $g\in G$. 
\end{lem}

\begin{proof}
Without loss of generality, we may assume $U=U^{-1}$ and $1_G \in U$. 

Let $\s:G \to \sym(V)$ be a discrete $(U^2,\eps|U|^{-1})$-sofic approximation to $G$. Let $D \subset G$ be the set of order 2 elements. Let $H \subset G$ be a subset such that for every $g \in G$ there is a unique element in the intersection $H \cap \{g,g^{-1}\}$. 

If $v \in V[\s,U^2]$ then $\s(1_G)^2 v = \s(1_G) v$. Thus $\s(1_G) v = v$. If also $g \in D \cap U^2$ is nontrivial then
\begin{eqnarray*}
\s(g)^2 v = \s(g^2)v = \s(1_G)v = v, \quad \s(g)v \ne \s(1_G)v = v.
\end{eqnarray*}
So there exists an element $\s'(g) \in \sym(V)$ with order 2 such that $\s'(g)v=\s(g)v$ for all $v\in V[\s,U^2]$. This defines $\s'(g)$ for all $g \in D \cap U^2$. Also define
$$\s'(g) = \left\{ \begin{array}{cc}
\textrm{identity} &  g \in \{1_G\} \cup (D \setminus U^2) \\
\s(g) & g \in H \setminus (D \cup \{1_G\})\\
\s(g^{-1})^{-1} & g^{-1} \in H \setminus (D \cup \{1_G\})\\
\end{array}\right.$$
This defines $\s'$ on all of $G$. Note $\s'(1_G)$ is the identity and $\s'(g)^{-1}=\s'(g^{-1})$ for all $g \in G$. Moreover, $\s'(g)v \in \{\s(g)v, \s(g^{-1})^{-1}v\}$ for all $v \in V[\s,U^2]$ and $g \in U^2$. 

It now suffices to show $\s'$ is a discrete $(U,\eps)$-sofic approximation. To prove this, let 
$$W = \{v \in V:~ \s(g)v \in V[\s,U^2] ~\forall g\in U\}.$$
We claim that $W \subset V[\s', U]$. 

To prove this we observe: if $v \in V[\s,U^2]$ and $g \in U^2$ then $\s'(g)v = \s(g)v = \s(g^{-1})^{-1}v$. Indeed 
$$\s(g^{-1})\s(g)v =\s(1_G)v = v.$$
Thus $\s(g)v = \s(g^{-1})^{-1}v$. Since $\s'(g)v \in \{\s(g)v, \s(g^{-1})^{-1}v\}$, it follows that $\s'(g)v=\s(g)v$. This proves the claim.

Now let $w \in W$ and $g,h \in U$. By definition of $W$, $\s(h)w  \in V[\s, U^2]$. Since $gh\in U^2$,
$$\s'(g)\s'(h)w = \s'(g)(\s(h)w) = \s(g)\s(h) w= \s(gh) w = \s'(gh)w.$$
Moreover, if $g \ne h$ then 
$$\s'(g)w = \s(g)w \ne \s(h)w = \s'(h)w.$$
This shows  $W \subset V[\s', U]$ as claimed.

By definition $W = \bigcap_{g \in U} \s(g)^{-1}V[\s, U^2]$. Since each $\s(g)$ is a permutation and $|V(\s,U^2)|\ge (1-\eps |U|^{-1})|V|$, this implies $|W| \ge (1-\eps)|V|$. Thus 
$$\#V[\s', U] \ge \# W \ge (1-\eps)\#V.$$
This shows $\s'$ is a discrete $(U,\eps)$-sofic approximation.

\end{proof}

\begin{thm}\label{T:discretegroup}
Let $G$ be a discrete countable group. Then $G$ is sofic as a discrete group if and only if $G$ is sofic in the sense of Definition \ref{D:sofic group}.
\end{thm}

\begin{proof}
Suppose $G$ is sofic as a discrete group. Let $d_G$ be a proper left-invariant metric on $G$. Recall that $B(\rho)$ denotes the open radius $\rho$ ball centered at the identity in $G$. It suffices to show that for every radius $\rho>0$ and $\eps>0$, there is a $(B(\rho), \eps)$-sofic approximation to $G$ (in the sense of Definition \ref{D:sofic group}).

Let $\s: G \to \sym(V)$ be a discrete $(B(2\rho), \eps)$-sofic approximation to $G$. By Lemma \ref{L:trivial}, we may assume $\s(1_G)$ is the identity permutation and $\s(g^{-1})=\s(g)^{-1}$ for all $g\in G$. 

For $p \in V$, define $\injrad(\s,p)$ to be the supremum of $\eta>0$ such that
\begin{enumerate}
\item  $\s(gh)^{-1}p= \s(h)^{-1}\s(g)^{-1}p$ for all $g,h \in B(\eta)$;
\item $\s(g)^{-1}p \ne \s(h)^{-1}p$ if $g,h \in B(\eta)$ with $g \ne h$.  
\end{enumerate}
So $p \in V[\s,B(\rho)]$ if and only if $\injrad(\s,p) \ge \rho$.

Let $\dom(\a)$ be the set of all $(p,g)$ in  $V \times G$ such that either $\injrad(\s,p) > d_G(1_G,g)$ or $\injrad(\s,\s(g)^{-1}p)>d_G(1_G,g)$. Define $\a:\dom(\a) \to V$ by $\a(p,g)=\s(g)^{-1}p$. We claim that $(V,\a)$ is a local $G$-space. It is immediate that Axioms 1, 2 and 4 of Definition \ref{D:partial} hold. 

To verify Axiom 3, suppose $(p,g), (\a(p,g),h), (p,gh) \in \dom(\a)$. Then 
 $$\a( \a(p,g), h) = \s(h)^{-1} \a(p,g) = \s(h)^{-1}\s(g)^{-1}p$$
 $$\a(p,gh) = \s(gh)^{-1}p.$$
 So we must show 
 \begin{eqnarray}\label{E:goal}
 \s(h)^{-1}\s(g)^{-1}p = \s(gh)^{-1}p.
 \end{eqnarray}
 
Because $(p,g), (\a(p,g),h), (p,gh) \in \dom(\a)$,
\begin{eqnarray*}
\textrm{ either } \injrad(\s,p)>d_G(1_G,g) &\textrm{ or }& \injrad(\s,\s(g)^{-1}p)>d_G(1_G,g), \\ 
 \textrm{ either } \injrad(\s,\s(g)^{-1}p)>d_G(1_G,h) &\textrm{ or }& \injrad(\s,\s(h)^{-1}\s(g)^{-1}p)>d_G(1_G,h), \\ 
 \textrm{ either } \injrad(\s,p)>d_G(1_G,gh) &\textrm{ or }& \injrad(\s,\s(gh)^{-1}p)>d_G(1_G,gh). 
 \end{eqnarray*}
 
Choose $q \in \{p, \s(g)^{-1}p, \s(h)^{-1}\s(g)^{-1}p, \s(gh)^{-1}p\}$ to maximize the injectivity radius $\injrad(\s,q)$. Note $\injrad(\s,q) > d_G(1_G,f)$ for all $f \in \{g,h,gh\}$. 
 
 If $q=p$ then (\ref{E:goal}) follows by definition of $\injrad(\s,p)$.  If $q = \s(g)^{-1}p$ then 
 $$\s(gh)\s(h)^{-1}\s(g)^{-1}p = \s(g)\s(g)^{-1}p = p$$
  by definition of $\injrad(\s,\s(g)^{-1}p)$ and the assumption $\s(h)^{-1}=\s(h^{-1})$. This also implies (\ref{E:goal}) by multiplying both sides by $\s(gh)^{-1}$. The other cases are similar. This verifies Axiom 3. 
 

  
It is immediate that $V[\s,B(2\rho)] \subset V[\a,B(\rho)]$.  So $\#V[\a,B(\rho)] \ge \#V[\s,B(2\rho)] \ge (1-\eps)\#V$. This proves $(V,\a)$ is $(B(\rho),\eps)$-sofic in the sense of Definition \ref{D:sofic group}. Since $\rho,\eps$ are arbitrary, this proves $G$ is sofic.

Now suppose $G$ is sofic in the sense of Definition \ref{D:sofic group}. Let $\rho>0$ be a radius and $\eps>0$. It suffices to show there exists $\s:G \to V$ such that $\s$ is a discrete $(B(\rho),\eps)$-sofic approximation.

By Theorem \ref{thm:injrad1}, there exists a sofic approximation $M=(M,\a)$ to $G$ such that
\begin{eqnarray}\label{E:injrad}
\frac{ \vol_{M}( \{p \in M:~ \injrad(M,p)>3\rho \}) }{\vol_{M}(M)} >1-\eps|B(\rho)|^{-1}.
\end{eqnarray}

Because $G$ is discrete, we choose $\Haar_G$ to be counting measure on $G$. Therefore, $\vol_M$ is counting measure on $M$. In particular, $M$ is finite.

By Lemma \ref{L:injective}, for $g \in G$, the map $\a(\cdot, g)$ is injective on its domain. Therefore, there exists a permutation $\s(g^{-1}) \in \sym(M)$ that agrees with $\a(\cdot, g)$ on its domain. So there is a map $\s:G \to \sym(M)$ such that $\s(g)p = \a(p,g^{-1})$ for all $(p,g^{-1}) \in \dom(\a)$. We claim that $(M,\s)$ is a $(B(\rho),\eps)$-discrete sofic approximation.

Let $D$ be the set of all $p \in M$ such that $\injrad(M, p.g) >2\rho$ for all $g \in B(\rho)$. By (\ref{E:injrad}), $|D| \ge (1-\eps)|M|$. So it suffices to show that $D \subset M[\s,B(\rho)]$. 

Let $p \in D$ and $g,h \in B(\rho)$. By the triangle inequality, $gh \in B(2\rho)$. Since $\injrad(M, p) >2\rho$, $\s(gh)p = \s(g)\s(h)p$. Moreover, if $g \ne h$ then $\s(g)p \ne \s(h)p$. These two claims imply $D \subset M[\s,B(\rho)]$ and so completes the proof.

 \end{proof}

\subsection{Stability of soficity under constructions}

\subsubsection{Inducing from a subgroup} \label{seg.induce}


In this section we prove that if $G$ contains a sofic lattice $\G \le G$ then $G$ is sofic as well. Moreover, if $\Si=\{V_i\}_{i \in \N}$ is a sofic approximation to $\G$ then there is an {\bf induced sofic approximation} $\Ind_\G^G(\Si) = \{\Ind_\G^G(V_i)\}_{i \in \N}$ to $G$. This is similar to the way that an action or representation of $\G$ can be induced to $G$. It depends apriori on a choice of fundamental domain $\Delta \subset G$. We will choose $\Delta$ to have some additional properties that will make it easier to prove that the induced map really is a sofic approximation. It seems likely that different fundamental domains lead to essentially the same induced sofic approximation but we make no effort to prove it. The next lemma gives a `nice' fundamental domain.

\begin{defn}
A {\bf lattice} is a subgroup $\G\le G$ such that, with the induced topology, $\G$ is discrete and $\G \backslash G$ has a finite $G$-invariant Borel measure. A {\bf fundamental domain} for $\G$ is a Borel set $\Delta \subset G$ such that $\sqcup_{g\in \G} g\Delta$ is a partition of $G$. 
\end{defn}

\begin{defn}
Let $X$ be a topological space. A collection $\{Y_i\}_{i \in I}$ of subsets $Y_i \subset X$ is {\bf locally finite} if for every $x \in X$ there exists an open neighborhood $O$ of $x$ in $X$ such that 
$$\#\{i \in I:~ Y_i \cap O \ne \emptyset \} < \infty.$$
This condition implies that for every compact $K \subset X$,
$$\#\{i \in I:~ Y_i \cap K \ne \emptyset \} < \infty.$$
\end{defn}

\begin{lem}\label{L:fund-domain}
Let $\G \le G$ be a lattice.  Then there exists a fundamental domain $\Delta$ for $\G$ such that the collection $\{g\Delta:~g\in \G\}$ is locally finite.

\end{lem}

\begin{proof}
Let $\pi:G \to \G \backslash G$ be the quotient map. Because $\G$ is discrete, for every $g \in G$ there exists an open pre-compact neighborhood $\tO$ of $g$ such that the restriction of $\pi$ to $\tO$ is injective. So there exists an open cover $\{O_i\}_{i\in I}$ of $\G \backslash G$ such that for each $i\in I$, there is a pre-compact open set $\tO_i \subset G$ such that $\pi$ restricted to $\tO_i$ is a homeomorphism onto $O_i$. Because $G$ is locally compact, after passing to a sub-cover if necessary, we may assume $\{O_i\}_{i\in I}$  is locally finite. This implies that the cover $\{g\tO_i\}_{g\in \G, i \in I}$ of $G$ is also locally finite. In fact, because $\tO_i$ is pre-compact and $\G$ is discrete, for any compact $K \subset G$ there are only finitely many $g\in \G$ with $g\tO_i \cap K \ne \emptyset$. On the other hand, there are only finitely many indices $i\in I$ with $\G \tO_i \cap K \ne \emptyset$ because this condition implies $\G K \cap O_i \ne \emptyset$ and $\{O_i\}_{i\in I}$ is locally finite.

Since $G$ is second countable, the index set $I$ is at most countable, so we may assume $I\subset \N$. Define
$$\Delta = \bigcup_{i \in I} \tO_i \setminus (\cup_{g\in \G} \cup_{j<i} g\tO_j).$$
Then $\Delta$ is a Borel fundamental domain. 


Let $x\in G$. Let $\tO_x \subset G$ be a pre-compact open neighborhood of $x$ such that the restriction of $\pi$ to $\tO_x$ is a homeomorphism onto its image. Because $\{g\tO_i\}_{g\in \G, i \in I}$ of $G$ is locally finite, there are only finitely many pairs $(g,i) \in \G \times I$ such that $g\tO_i \cap \tO_x  \ne \emptyset$.  

We claim that $\tO_x$ intersects at most finitely many $\G$-translates of $\Delta$. To see this, let $g \in \G$ and suppose $g \Delta \cap \tO_x \ne \emptyset$. Since $\Delta \subset \cup_{i\in I} \tO_i$, this implies the existence of $i\in I$ with $g\tO_i \cap \tO_x \ne \emptyset$. So by the previous paragraph, there are only finitely many $g \in \G$ with $g \Delta \cap \tO_x \ne \emptyset$. This finishes the lemma. 
\end{proof}

\begin{thm}\label{T:soficlattice}
Let $\G \le G$ be a lattice where $G$ is an lcsc group. If $\G$ is sofic then $G$ is sofic. Moreover, for every sofic approximation $\Si=\{V_i\}_{i\in I}$ to $\G$ there is an induced sofic approximation $\Ind_\G^G(\Si)  = \{ \Ind_\G^G(V_i)\}_{i\in I}$ to $G$ determined only by $\Si$ and a choice of fundamental domain $\Delta$ for $\G$ satisfying Lemma \ref{L:fund-domain}.
\end{thm}

\begin{proof}
Let $\Delta$ be a fundamental domain for $\G$ satisfying Lemma \ref{L:fund-domain}.  Define a section $\s:\G \backslash G \to \Delta$  by  $\G \s(\G g) = \G g$. This is well-defined because $\Delta$ is a fundamental domain. Define $c:\G\backslash G \times G \to \G$ by
$$c(\G h, g) = \s(\G h) g \s(\G hg)^{-1}.$$ 
Then $c$ satisfies the cocycle equation $c(\G h, g) c(\G hg, k) = c(\G h, gk)$ for any $g,h,k$. 

If $K \subset G$ then we write $c(\G h, K) = \{ c(\G h, k):~k \in K\} \subset \G$. We claim that if $K$ is compact then $c(\G h, K)$ is finite. To see this, observe that $c(\G h, k) \Delta \cap \s(\G h) K \ne \emptyset$ (for all $k\in K$). In fact, 
$$\s(\G h)k = c(\G h, k) \s(\G hk) \in c(\G h, k) \Delta \cap \s(\G h) K.$$
  Because the collection $\{g \Delta\}_{g\in \G}$ is locally finite and $\s(\G h) K$ is compact, this implies the claim: $c(\G h, K)$ is finite.

Let $\Si=\{V_i\}_{i\in I}$ be a sofic approximation to $\G$ (in the sense of Definition \ref{D:sofic group}). We will denote the partial action of $\G$ on $V_i$ by $v.g$ for $v \in V_i, g\in \G$ whenever this is well-defined.

For a warm-up exercise, let's handle the special case in which there is a finite-index subgroup $H_i \le \G$, $V_i=H_i \backslash \G$ and the partial action of $\G$ on $V_i$ is the usual action by right-translation. In that case, $G$ acts on $V_i \times \G \backslash G$ by $(v, \G h).g = (v.c(\G h, g), \G hg)$. Define a $G$-equivariant Borel isomorphism 
$$\Phi: V_i \times \G \backslash G \to H_i \backslash G, \quad \Phi(H_i g, \G h) = H_i g \s(\G h).$$
The action of $G$ on $V_i \times \G \backslash G$ is not continuous with respect to the product topology. So we re-topologize $V_i \times \G \backslash G$ by pulling back the topology on $H_i \backslash G$. 

Because the quotient map $G \to H_i \backslash G$ is a covering space map, the topology on $H_i \backslash G$ is such that the open sets in $H_i \backslash G$ are images of open sets in $G$. So the new topology on $V_i \times \G \backslash G$ has a basis of open sets given by sets of the form $p.O$  where $p \in V_i \times \G \backslash G$ and $O \subset G$ is an open neighborhood of the identity. 


%

Now for the general case. For $(v, \G h)\in V_i \times \G \backslash G $ and $g \in G$ we write
$$\a'((v,\G h), g) = (v,\G h).g := (v.c(\G h, g), \G hg)$$
whenever this is well-defined. Observe that by Axiom 3 (applied to $V_i$), if $p \in V_i \times \G \backslash G$ and $g,h \in G$ are such that $p.g, p.g.h$ and $p.gh$ are all well-defined then $p.g.h = p.gh$.

We say that a subset $O \subset G$ is {\bf good} for a point $p \in V_i \times \G \backslash G$ if
\begin{itemize}
\item $O$ is an open neighborhood of the identity,
\item for every $h_1,h_2 \in O$, $p.h_1$ and $p.h_1.h_1^{-1}h_2$ are well-defined, and
\item the map which sends $g \in O$ to $p.g$ is injective.
\end{itemize}
Note that if $O$ is good for $p$, then every open subset of $O$ containing the identity is also good for $p$. Let $M_i \subset V_i \times \G \backslash G$ be the set of all points $p$ for which there exists a good set $O \subset G$.  If $p \in M_i$ and $O$ is good for $p$ then we write $p.O =\{p.g:~ g \in O\}$. 

\noindent {\bf Claim 1}. If $O$ is good for $p$ and $g \in O$ then $g^{-1}O$ is good for $p.g$. Moreover, $p.O = p.g.g^{-1}O$.
\begin{proof}[Proof of Claim 1]
Let $h_1,h_2 \in g^{-1}O$. We must show $p.g.h_1$ and $p.g.h_1.h_1^{-1}h_2$ are well-defined. 

Because $O$ is good for $p$ and $g, gh_1 \in O$, it follows that $p.g.h_1= p.g.g^{-1}gh_1$ is well-defined.  Moreover, $p.g$ and $p.gh_1$ are well-defined. Therefore, $p.gh_1 = p.g.h_1$. So it now suffices to show $p.gh_1.h_1^{-1}h_2$ is well-defined. But this follows from goodness of $O$ because $gh_1, gh_2 \in O$ and $p.gh_1.(gh_1)^{-1}gh_2 = p.gh_1.h_1^{-1}h_2$. This proves the first statement.

Note that if $O$ is good for $p$ and $g,h \in O$ then $p.g.g^{-1}h = p.h$ (because both sides are well-defined). The second statement follows.
\end{proof}
Claim 1 implies that $p.O \subset M_i$. 

\noindent {\bf Claim 2}.
The collection of subsets of $M_i$ of the form $p.O$ (where $p \in M_i$ and $O$ is good for $p$) is a base for a topology on $M_i$.
\begin{proof}[Proof of Claim 2]
 It suffices to show that if $O$ is good for $p$ and $U$ is good for $q$ and $r \in p.O \cap q.U$ then there is a good set $W \subset G$ for $r$ such that $r.W \subset p.O \cap q.U$. Let $r=p.g = q.h$ for some $g \in O$, $h\in U$. Then $g^{-1}O$ and $h^{-1}U$ are good for $r$ by Claim 1. It follows that $W=g^{-1}O \cap h^{-1}U$ is also good for $r$. Moreover, $r.W = p.g.(g^{-1}O \cap h^{-1}U) \subset p.O$. Similarly, $r.W \subset q.U$. 
\end{proof}
From now on, we consider $M_i$ with the topology induced by sets of the form $p.O$ as above. Claim 2 implies that  if $O$ is good for $p$ then the map $g \mapsto p.g$ from $O$ into $M_i$ is a homeomorphism onto an open subset of $M_i$. In particular, $M_i$ is locally compact. 

We claim $M_i$ is second countable.  Because $G$ is second countable, there is a countable base $\cB$ for the topology on $\G \backslash G$. Now suppose $p\in M_i$ and $U \subset G$ is good for $p$. Let $p=(v,\G h)$. Because $\cB$ is a base and $\G h U$ is open in $\G \backslash G$, there exists an open subset $U' \subset U$ with $1_G \in U'$ such that $\G h U' \in \cB$. Since $U'$ is good for $p$, $p.U'$ is open in $M_i$. Thus, the collection of sets of the form $p.U'$ is a countable base for the topology on $M_i$.




Let $\dom(\a)$ be the set of all $(p,g) \in M_i \times G$ such that there is an open neighborhood of $(p,g)$ in $M_i\times G$ on which $\a'$ is well-defined. Let $\a$ be the restriction of $\a'$ to $\dom(\a)$. 
 
 \noindent {\bf Claim 3}.
$\a$ is continuous. Moreover, if $(q,f)\in \dom(\a)$ then there is an open neighborhood $W$ of $(q,f)$ such that $\a(W)$ is open in $M_i$. 
\begin{proof}[Proof of Claim 3]
By Claim 2, it suffices to prove: if $p \in M_i$ and $O \subset G$ is good for $p$ then $\a^{-1}(p.O)$ is open in $M_i\times G$. So let $(q,f) \in \a^{-1}(p.O)$. 
Let $g \in O$ be such that $p.g=q.f$. By Claim 1, $g^{-1}O$ is good for $p.g=q.f$.


Let $U_1 \subset G$ be a set which is good for $q$ so that 
$$(q.U_1\times \{f\}) \cup (\{q\}\times U_1f) \subset \dom(\a)$$
and $f^{-1}U_1f \subset g^{-1}O$. We claim that $q.U_1f \subset p.O$. To see this, let $h \in U_1$. Then $q.hf = q.f.f^{-1}hf$ because both sides are well-defined (by Claim 1 applied to $q.f=p.g$). Since $q.f.f^{-1}hf=p.g.f^{-1}hf \subset p.O$, this proves $q.hf \in p.O$. Since $h$ is arbitrary,  $q.U_1f  \subset p.O$.

After choosing $U_1$ smaller if necessary, we may assume the closure of $q.U_1f$ is contained in $p.O$. Therefore, there is an open neighborhood $U_2\subset G$ of the identity such that $q.U_1f.U_2 \subset p.O$, 
$$(q.U_1\times fU_2) \cup (\{q\}\times U_1fU_2) \cup (q.U_1f \times U_2) \subset \dom(\a)$$
and $f^{-1}U_1f U_2 \subset g^{-1}O$. It follows that $W:=q.U_1\times fU_2 \subset M_i \times G$ is an open neighborhood of $(q,f)$. Moreover, $\a(W) = q.U_1.fU_2 = q.U_1fU_2 = q.U_1f.U_2$ since these are all well-defined. By assumption on $U_2$, this shows $\a(W) \subset p.O$. This shows $\a$ is continuous. 

Note $\a(W) = q.U_1fU_2 = q.f.f^{-1}U_1fU_2$ because both sides are well-defined. Because $f^{-1}U_1fU_2 \subset g^{-1}O$ which is good for $p.g=q.f$, it follows that $\a(W)$ is open in $M_i$.



\end{proof}








We claim the space $M_i$ with the partial action defined above is a local $G$-space. Axioms 1, 3 and 4 are immediate. To establish Axiom 2, suppose that $(p,g)\in \dom(\a)$. We have to show $(p.g,g^{-1})\in \dom(\a)$. There is an open subset $W \subset \dom(\a)$ containing $(p,g)$. For each $(q,f) \in W$, $q.f.f^{-1}$ is well-defined. Moreover, the set $\{(q.f, f^{-1}):~ (q,f) \in W\}$ is an open neighborhood of $(p.g, g^{-1})$. This implies Axiom 2.


Let $U \subset G$ be a precompact open neighborhood of the identity with $U=U^{-1}$ and $\eps>0$. We will show that if $i$ is sufficiently large then $M_i$ is a $(U,\eps)$-sofic approximation.

Given $F \subset \G$, let $\Omega(F)$ be the set of all $\G h \in \G \backslash G$ such that for every $g_1,g_2 \in U^3$ with $g_1g_2 \in U$, $c(\G h g_1, g_2) \in F$. Because $U$ is precompact, $c(\G h, U)$ is finite for every $h$. So there exists a finite set $F \subset \G$ such that 
$$\vol_{\G \backslash G}(\Omega(F)) > (1-\eps/2) \vol_{\G \backslash G}(\G \backslash G).$$
After choosing $F$ larger if necessary, we may assume $1_G \in F$ and $F=F^{-1}$. 

Because $\Si$ is a sofic approximation, there exists $I$ such that $i>I$ implies $V_i$ is an $(F^2,\eps/2)$-sofic approximation to $\G$. We claim that if $i>I$ then $M_i$ is a $(U,\eps)$-sofic approximation. Because 
\begin{eqnarray*}
\vol_{M_i}(V_i[F^2] \times \Omega(F)) &=& \vol_{V_i}(V_i[F^2]) \times \vol_{\G \backslash G}(\Omega(F)) \ge (1-\eps/2)^2|V_i|\vol_{\G \backslash G}(\G \backslash G)\\
 &\ge&  (1-\eps/2)^2\vol_{M_i}(M_i),
 \end{eqnarray*}
it suffices to show $M_i[U] \supset V_i[F^2] \times \Omega(F)$. 

First we claim that if $p=(v,\G h) \in  V_i[F^2] \times \Omega(F)$ then $p.U.U^2$ is well-defined. This is equivalent to the well-definedness of  $v.c(\G h, g_1).c(\G h g_1, g_2)$ for all $g_1 \in U$ and $g_2 \in U^2$. Because $\G h \in \Omega(F)$, $c(\G h, g_1)$ and $c(\G h, g_1g_2)$ are in $F$. By the cocycle equation, $c(\G h g_1, g_2) = c(\G h, g_1)^{-1}c(\G h, g_1g_2) \in F^2$. Since $v \in V_i[F^2]$, $v.c(\G h, g_1).c(\G h g_1, g_2)$ is well-defined. 

Next we show that $U$ is good for $p$. The first condition of `good' is trivial and the second condition holds by the paragraph above. To check the third condition, suppose $g_1, g_2 \in U$ and $p.g_1=p.g_2$. Because $c(\G h, g_i) \in F$ ($i=1,2$), and $v \in V_i[F^2]$, the condition $p.g_1=p.g_2$ implies $c(\G h, g_1) = c(\G h, g_2)$. Equivalently,
$$\s(\G h) g_1 \s(\G hg_1)^{-1} = \s(\G h) g_2 \s(\G hg_2)^{-1}.$$
Because we also have $\G h g_1=\G h g_2$, it follows that $g_1=g_2$. This verifies $U$ is good for $p$. In particular, $p\in M_i$.

Next we claim that for all $g\in U$, $(p,g) \in \dom(\a)$. It suffices to show $p.U \times gU$ is an open neighborhood of $(p,g)$ in $\dom(\a')$. This is implied by the paragraph above, which shows $p.U.gU$ is well-defined. So $V_i[F^2] \times \Omega(F) \times U \subset \dom(\a)$. 


Next we show that the map which sends $g\in U$ to $p.g$ is a homeomorphism onto an open subset of $M_i$. Because $U$ is good for $p$, this map is well-defined and injective. By claim 3, it is continuous.  Because $U$ is pre-compact, this implies the map is a homeomorphism onto an open subset. Thus $M_i[U] \supset V_i[F] \times \Omega(F)$ as claimed.

Because $U,\eps$ are arbitrary, this shows $\{M_i\}$ is a sofic approximation to $G$.




\end{proof}

\begin{cor}\label{C:lattices}
The following groups are sofic: semi-simple Lie groups (e.g. $\rm{SL}(n,\R), \rm{SO}(n,1)$ etc), the automorphism group of a regular tree. 
\end{cor}

\begin{proof}
These groups admit residually finite lattices. Since residual finiteness implies soficity, the corollary follows from Theorem \ref{T:soficlattice}.
\end{proof}

\subsubsection{Restricting to a subgroup}

It is well-known that if a countable group $\Ga$ is sofic then all of its subgroups are sofic. Indeed, given a sofic approximation $\{\s_i:\G \to \sym(V_i)\}$ one can restrict the maps $\s_i$ to a subgroup $\La$ to obtain a sofic approximation to $\La$. This argument fails in the general setting of locally compact groups because, if $H\le G$ then Haar measure on $H$ might be singular to Haar measure on $G$. However, the following gives a positive result. 

\begin{prop}\label{P:open}
Let $G$ be a locally compact sofic group. If $H \le G$ is an open subgroup then $H$ is sofic.
\end{prop}

\begin{proof}
Let $(M,\a)$ be a local $G$-space. Define a partial action $\a_H$ by $\a_H: \dom(\a_H) \to M$, $\a_H(p,g)=\a(p,g)$ where $\dom(\a_H) = \dom(\a) \cap (M \times H)$. 

We claim that $(M,\a_H)$ is a local $H$-space. To see this, let $p \in M$. By Axiom 4 of Definition \ref{D:partial}, there is an open neighborhood $O_p$ of $1_G$ in $G$ such that the restriction of $\a(p,\cdot)$ to $O_p$ is a homeomorphism onto an open neighborhood of $p$ in $M$. Because $H$ is open, the restriction of $\a(p,\cdot)$ to $O_p \cap H$ is also a homeomorphism onto an open neighborhood of $p$ in $M$. This shows Axiom 4. The other Axioms are immediate. 


Let $U$ be an open neighborhood of $1_H$ in $H$ and $\eps>0$. Because $H$ is open in $G$, $U$ is also an open neighborhood of $1_G$ in $G$. By definition, $M[U,\a] = M[U,\a_H]$. Because Haar measure on $H$ equals Haar measure on $G$ restricted to $H$, the choice $\vol_M$ does not depend on whether we consider $M$ to be a local $G$-space or a local $H$-space. So  if $(M,\a)$ is a $(U,\eps)$-sofic approximation to $G$, then $(M,\a_H)$ is also a $(U,\eps)$-sofic approximation to $H$.

\end{proof}


\section{Open problems}

\subsection{Which groups are sofic?}

\begin{problem}
Are all unimodular lcsc groups sofic? For example, the Neretin group is a unimodular lcsc group without lattices \cite{MR2881324}. Is it sofic? 
\end{problem}

\begin{problem}
 If $G$ is linear and unimodular then is $G$ sofic? By Mal'cev's Theorem \cite{malcev-1940} if $G$ is finitely generated and linear then it is residually finite. Because increasing unions of sofic groups are sofic, if $G$ is countable and linear then it is sofic. 
\end{problem}

\begin{problem}
If $G$ is connected and sofic then is its universal cover sofic?
\end{problem}

\begin{problem}
Suppose $G$ is a connected unimodular Lie group and let $S \le G$ be its solvable radical. If $G/S$ is sofic then is $G$ sofic? 
\end{problem}

\begin{problem}
Permanence properties for discrete countable sofic groups have been studied in \cite{CHR14, MR3795485, MR3934790,MR4140873} for example. These papers concern graph products, wreath products (restricted and unrestricted) and semi-direct products respectively. Are there analogs of these results for locally compact sofic groups?

\end{problem}

\begin{problem}
Suppose $G$ is non-unimodular lcsc group and let $\d:G \to \R_{>0}$ denote the modular homomorphism. Let $\hG = \R \rtimes G$ denote the semi-direct product with group law
$$(t, g)(s,h) = (t+\d(g)s, gh).$$
Then $\hG$ and $\Ker(\d)$ are unimodular groups. If $\hG$ is sofic then is $\Ker(\d)$ sofic? If $\Ker(\d)$ is sofic then is $\hG$ sofic?
\end{problem}

\subsection{Group rings}

\begin{problem}
If $G$ is sofic then is its group von Neumann algebra Connes-embeddable? Elek and Szabo proved the answer is `yes' in the case of discrete countable groups \cite{elek-szabo-2005}. 
\end{problem}

\begin{problem}
The algebraic eigenvalue conjecture of J.Dodziuk, P.Linnell, V.Mathai, T.Schick and S.Yates \cite{MR1990479} posits that if $\G$ is a discrete group and $A \in M_n(\Z\G)$ (the ring of $n\times n$ matrices with values in the group ring $\Z\G$) and $\l(A)$ is the corresponding operator on $\ell^2\G^{\oplus n}$ then all eigenvalues of $\l(A)$ are algebraic integers. This was proven true for sofic groups by  A. Thom \cite{MR2417890}. Is there an analogous statement for locally compact sofic groups?
\end{problem}

\begin{problem}[Kaplansky's Direct Finiteness Conjecture]
A ring $R$ is said to be {\bf directly finite} if $xy=1$ implies $yx=1$ for all $x,y \in R$. Kaplansky conjectured that if $G$ is a countable group and $k$ is field then the group ring $kG$ is directly finite. This is known as Kaplansky's Direct Finiteness Conjecture. If $G$ is sofic then as explained in Problem \ref{P:Gottschalk} below, it satisfies Gottschalk's surjunctivity conjecture. This immediately implies $kG$ is directly finite if $k$ is a finite field. The general case follows because all fields are embeddable into ultraproducts of finite fields. See \cite{capraro-lupini} for details. Is there an analogous statement in the setting of locally compact groups?
\end{problem}

\subsection{Actions}

\begin{problem}[Gottschalk's Surjunctivity Conjecture]\label{P:Gottschalk}
Suppose $X$ is a compact Hausdorff space, $G$ is a topological group and $G \times X \to X$ is a jointly continuous action. This action is said to be {\bf surjunctive} if every continuous injective $G$-equivariant map $\phi:X \to X$ is surjective. 

Gottschalk conjectured that if $G$ is discrete and $A$ is a finite set then the full shift $G \cc A^G$ is surjunctive where $A^G=\{x:G\to A\}$ has the topology of pointwise convergence and $G$ acts on $A^G$ by
$$(gx)(f)=x(g^{-1}f).$$
M. Gromov proved that sofic groups satisfy Gottschalk's conjecture \cite{MR1694588}. His proof was simplified and made more accessible by B. Weiss \cite{weiss-2000}. It was re-proven by Kerr and Li using sofic topological entropy \cite{kerr-li-variational}. Is there an analog of Gottschalk's conjecture in the locally compact setting? 

\end{problem}

\begin{problem}
Sofic approximations have been used to define invariants of actions of discrete countable groups on probability spaces, compact topological spaces and Banach spaces. These invariants include sofic measure entropy \cite{bowen-jams-2010}, topological sofic entropy \cite{kerr-li-variational}, sofic mean dimension \cite{MR3077882}, sofic mean length \cite{MR3993930} and $\ell^p$ dimension \cite{hayes-lp-dimension-1}. These invariants generalize classical invariants of $\Z$-actions. For an introduction to sofic entropy, see \cite{bowen-icm}. This motivates the problem: generalize these invariants to actions by locally compact groups. This problem is open except for the fact that Sukhpreet Singh's thesis generalized some of the foundational results of sofic entropy theory to locally compact groups. He has no plans to publish his thesis, but copies are available upon request to the first author. \end{problem}

\subsection{Sofic approximations}

\begin{problem}[Amenable groups]
Elek and Szabo proved a structure theorem for sofic approximations of discrete amenable groups in \cite{MR2823074}. It states that there is essentially only one sofic approximation to a discrete amenable group (up to asymptotically vanishing perturbations and taking disjoint copies) which is given by a F\o lner sequence. Is there an analogous statement in the setting of locally compact groups?
\end{problem}

\begin{problem}[Flexible stability]
Let us say that a sofic group $G$ is {\bf flexibly stable} if for every $\d>0$ there are pre-compact open $U \subset G$ and $\eps>0$ such that if $M$ is an $(U,\eps)$-sofic approximation to $G$ then there exist lattice subgroups $\G_1,\ldots, \G_k \le G$, an open subspace $M' \subset M$, an open subspace $X' \subset X$ where $X:=\sqcup_{i=1}^k \G_i \backslash G$ is the disjoint union and a homeomorphism $\Phi:M' \to X'$ such that
\begin{itemize}
\item $\Phi(p.g)=\Phi(p)g$ whenever both sides are defined;
\item $\vol(M') \ge (1-\d)\vol(M)$;
\item $\vol(X') \ge (1-\d)\vol(X)$. 
\end{itemize}
This implies that any sofic approximation to $G$ can, by a small perturbation, be changed into an approximation by a disjoint union of coset spaces. It follows from Benjy Weiss's results in \cite{weiss-monotileable} and Elek and Szabo's structure theorem for sofic approximations of amenable groups \cite{MR2823074} that residually finite amenable discrete groups are flexibly stable. It is a folklore result that free groups are flexibly stable. In recent work, it has been shown that surface groups are flexibly stable \cite{2019arXiv190107182L}. This motivates the following questions: is PSL$(2,\R)$ flexibly stable? $\Aut(T_d)$? SO$(3,1)$? $\R^2 \rtimes \rm{SL}(2,\R)$? Are amenable unimodular groups that admit residually finite lattices flexibly stable? If $G$ is flexibly stable then are all lattice subgroups of $G$ flexibly stable?
\end{problem}

\begin{problem}[Property (T) and expanders]
If $M$ is a complete Riemannian manifold with finite volume then the {\bf Cheeger constant} of $M$ is
$$h(M) = \inf_{K \subset M} \frac{\textrm{area}(\partial K)}{\vol(K)}$$
where the infimum is over all compact smooth sub-manifolds $K \subset M$ with $0< \vol(K) \le \vol(M)/2$ \cite{MR0402831}.

Suppose $G$ is a connected Lie group with property (T). It is well-known that there is a positive lower bound $\eps_0>0$ on the Cheeger constants of coset spaces $\G \backslash G$. That is $h(\G \backslash G) \ge \eps_0$ for all lattices $\G \le G$. With this in mind, we conjecture that there exists $\eps'_0>0$ such that  for every $\d>0$ there exist a pre-compact open set $U \subset G$ and $\eps>0$ such that if $M$ is any $(U,\eps)$-sofic approximation to $G$ then there exists a smooth submanifold $M' \subset M$ satisfying 
\begin{itemize}
\item $\vol(M') \ge (1-\d)\vol(M)$;
\item every connected component of $M'$ has Cheeger constant $\ge \eps'_0$. 
\end{itemize}
A similar conjecture for discrete (T) groups by the first author was proven by Gabor Kun \cite{kun-2016}. Maybe there is a common generalization to all locally compact (T) groups?
\end{problem}

\begin{problem}
The {\bf sofic dimension} of a countable discrete group measures the growth rate of the number of sofic approximations to the group \cite{dykema-2014, MR3314839}.  It is a combinatorial version of the free entropy dimension \cite{MR1371236}. Moreover it admits a natural formula with respect to free product with amalgamation over an amenable group. Are there analogs of these results in the locally compact setting?
\end{problem}


\begin{problem}
A countable discrete group $G$ is sofic if and only if it embeds into a metric ultraproduct of finite symmetric groups \cite{elek-szabo-2005, pestov-sofic-survey}. Is there an analogous fact for locally compact sofic groups? Note that $\textrm{SL}(2,\R)$ is sofic (because it admits a residually finite lattice) but it does not continuously embed into a metric ultraproduct of compact groups. This is because $\textrm{SL}(2,\R)$ does not admit a proper bi-invariant metric.
\end{problem}

\subsection{Groupoids and measured equivalence relations}

\begin{problem}
Soficity was generalized to discrete measured equivalence relations and groupoids in \cite{MR2566316, paunescu-2011, dykema-2014}. Can this theory be generalized to measured equivalence relations and groupoids with locally compact leaves? \end{problem}

\begin{problem}
It might be possible to reduce soficity of a non-discrete lcsc group to soficity of a related discrete measured equivalence relation. The latter notion was introduced in \cite{MR2566316}. 

It is well-known that a discrete countable group $G$ is sofic if there exists an essentially free action of $G$ on a standard probability space $(X,\mu)$ such that the orbit-equivalence relation is sofic. It seems likely that this fact generalizes to locally compact groups as follows. 

Suppose $G$ acts on a standard probability space $(X,\mu)$ preserving the measure. For simplicity, let us assume the action is essentially free. By \cite{FHM78} there is a complete lacunary section $S \subset X$. This means that $S$ is Borel, $GS$ is conull in $X$ and there is an open neighborhood $U \subset G$ of the identity such that $Ux \cap S =\{x\}$ for all $x \in S$. In particular, if $\cR^X_G =\{(x,gx):~x\in X, g\in G\}$ is the orbit-equivalence relation of the $G$ action and $\cR^S_G:=\cR_G \cap (S\times S)$ then $\cR^S_G$ is an equivalence relation on $S$ with countable classes. 

If $G$ is non-discrete then $\mu(S)=0$. In spite of this, there is a natural measure, denoted $\nu$, on $S$ which behaves as if it were $\mu$ conditioned on $S$. The measure is defined by
$$\nu(A) = \frac{\mu(VA)}{\Haar_G(V)}$$
where $A \subset S$ is any Borel set and $V\subset G$ is a symmetric open neighborhood of $1_G$ in $G$ such that $V^2 \subset U$. This does not depend on the choice of $V$. This is explained in \cite{avni-2010} for example. 

We conjecture: if there exists a pmp action $G \cc (X,\mu)$ and a section $S \subset X$ as above such that the discrete measured equivalence relation $(S,\nu,\cR^S_G)$ is sofic then $G$ is sofic. This might give an approach to proving the group von Neumann algebras of sofic groups are Connes-embeddable (see \cite{MR2566316} for related results in the discrete case).

\end{problem}

\begin{problem}
The previous problem gave a sufficient condition for soficity. There is a related equivalent condition. It is well-known that a discrete countable group $G$ is sofic if and only if for every Bernoulli shift action of $G$ the associated measured equivalence relation is sofic. It seems likely that this fact generalizes to locally compact groups as follows.

Consider the Poisson point process on a non-discrete lcsc group $G$ with intensity measure equal to left-Haar measure on $G$. Because $G$ is non-discrete, we can consider the law of this process to be a $G$-invariant probability measure $\mu$ on the space $\Om$ of discrete closed subsets of $G$. Let $\Om_1 \subset \Om$ be the set of discrete closed subsets $\om \subset G$ with $1_G \in \om$. Even though $\mu(\Om_1)=0$, there is a natural probability measure $\nu$ on $\Om_1$ that intuitively represents $\mu$ conditioned on $\Om_1$ (this exists even though $\Om_1$ is not a lacunary section). Define an equivalence relation $\cR$ on $\Om_1$  by $(\om_1,\om_2) \in \cR \Leftrightarrow \exists g \in G$ such that $g\om_1=\om_2$. Then $\cR$ is discrete and $\nu$-preserving. We conjecture that $G$ is sofic if and only if this measured equivalence relation is sofic.
\end{problem}

\appendix

\bibliography{biblio}
\bibliographystyle{alpha}

\end{document}